\documentclass[11pt]{amsart} 
\usepackage{amsthm,amsmath,amssymb,latexsym}
\usepackage{hyperref}
\hypersetup{backref, pdfpagemode=FullScreen, colorlinks=true}
\begin{document}
\oddsidemargin=18pt \evensidemargin=18pt
\newtheorem{theorem}{Theorem}[section]
\newtheorem{definition}[theorem]{Definition}
\newtheorem{proposition}[theorem]{Proposition}
\newtheorem{lemma}[theorem]{Lemma}
\newtheorem{remark}[theorem]{Remark}
\newtheorem{corollary}[theorem]{Corollary}
\newtheorem{question}{Question}
\newtheorem{fact}{Fact}[section]
\newtheorem{claim}[theorem]{Claim}

\newtheorem{notation}{Notation}[section]

\def\bem{\begin{remark}\upshape}
\def\ebem{\end{remark}}
\def\nota{\begin{notation}\upshape}
\def\enota{\end{notation}}
\def\defn{\begin{definition}\upshape}
\def\edefn{\end{definition}}
\def\thm{\begin{theorem}}
\def\ethm{\end{theorem}}
\def\lmm{\begin{lemma}}
\def\elmm{\end{lemma}}
\def\qed{\hfill$\quad\Box$}
\def\pr{\par\noindent{\em Proof: }}
\def\prcl{\par\noindent{\em Proof of Claim: }}
\def\cor{\begin{corollary}}
\def\ecor{\end{corollary}}
\def\frag{\begin{question}\upshape}
\def\efrag{\end{question}}
\def\prop{\begin{proposition}}
\def\eprop{\end{proposition}}
\def\cl{\par\noindent{\bf Claim: }\em}
\def\ecl{\par\noindent\upshape}
\def\bsp{\begin{example}}
\def\ebsp{\end{example}}\def\lcm{{\mathcal lcm}}
\def\b{{\mathbf{b}}}
\def\G{{ \mathbb G}}
\def\Z{{ \mathbb Z}}
\def\N{{ \mathbb N}}
\def\R{{ \mathbb R}}
\def\Q{{\mathbb Q}}
\def\D{{\mathbb D}}
\def\I{{\mathcal I}}
\def\Ca{{\mathcal C}}
\def\T{{\mathcal T}}
\def\A{{\mathcal A}}
\def\ki{{\bar K}}
\def\L{{\mathcal L}}
\def\M{{\mathcal M}}
\def\O{{\mathcal O}}
\def\P{{\mathcal P}}
\def\U{{\mathcal U}}
\def\u{{\mathcal u}}
\def\Rs{{\mathcal R}}
\def\Fa{{\mathcal F}}
\def\l{{\ell}}
\def\si{{\sigma}}
\def\lb{{\lbrack\vert}}
\def\rb{{\rbrack\vert}}
\def\si{{\sigma}}
\newcommand{\bigdoublewedge}{%
  \mathop{
    \mathchoice{\bigwedge\mkern-15mu\bigwedge}
               {\bigwedge\mkern-12.5mu\bigwedge}
               {\bigwedge\mkern-12.5mu\bigwedge}
               {\bigwedge\mkern-11mu\bigwedge}
    }
}

\newcommand{\bigdoublevee}{%
  \mathop{
    \mathchoice{\bigvee\mkern-15mu\bigvee}
               {\bigvee\mkern-12.5mu\bigvee}
               {\bigvee\mkern-12.5mu\bigvee}
               {\bigvee\mkern-11mu\bigvee}
    }
}


\title[Fractional Parts of Additive Subgroups]{Fractional Parts of Dense Additive Subgroups of Real Numbers}

\author{Luc B\'elair}
\address{\hskip-\parindent
Luc B\'elair\\
LACIM, D\'epartement de math\'ematiques \\
    Universit\'e du Qu\'ebec - UQAM\\
    C.P. 8888 succ. Centre-ville\\
    Montr\'eal (Qu\'ebec) H3C 3P8, 
    Canada.}
\email{belair.luc@uqam.ca}
\author{Fran\c coise Point$^{\dagger}$}
\address{\hskip-\parindent
  Fran\c coise Point\\
  D\'epartement de Math\'ematique\\
  Universit\'e de Mons, Le Pentagone\\
  20, Place du Parc, B-7000 Mons, Belgium.}
\email{point@math.univ-paris-diderot.fr}
\thanks{${\;}^{\dagger}$Research Director at the "Fonds de la Recherche
  Scientifique-FNRS".}

\begin{abstract}
  Given a dense additive subgroup $G$ of $\mathbb R$ containing $\mathbb Z$, we consider its intersection $\mathbb G$ with the interval $[0,1[$ with the induced order and the group structure given by addition modulo $1$. We axiomatize the  theory of $\mathbb G$ and show it is model-complete, using a Feferman-Vaught type argument. We show that any sufficiently saturated model decomposes into a product of a {\it standard} part and two ordered semigroups of infinitely small and infinitely large elements.
\end{abstract}
\maketitle                   

\section{Introduction}
\par  Motivated by the possibility of  improving the efficiency of timed-automata, F. Bouchy, A. Finkel and J. Leroux showed that any definable set in the structure  $(\R,\Z,+,-,0,1,<)$ can be written as a finite union of sums of definable subsets of the form $Z+D$, where $Z$ is a subset definable in $(\Z,+,0,<)$ and where $D$ is definable in $\mathbb D:=([0,1[, +_1,-_1, 0, <)$, the group of decimals with addition modulo $1$ and the induced order (\cite[Theorem 7]{BFL}). 
The first-order theories of   $(\Z,+,0,<)$ and $(\mathbb R, +, 0, <)$ are well known, the first one being essentially Presburger arithmetic and the other being ordered divisible abelian groups, and that of $(\R,\Z,+,-,0,1,<)$ is basically known (see \cite{W}, but also \cite{boigelot-rassart-wolper}, \cite{miller2001}, \cite{dolich-goodrick}). It seems natural to complete this picture by looking more closely at the first order theory of $\mathbb D$. It is interpretable in $(\mathbb R, +, 0, 1, <)$, so in particular decidable. In our preceding paper \cite{belair-point2014}, we gave an axiomatization of $\mathbb D$ and showed it admits quantifier elimination by adding some natural unary definable functions. In fact it  itself admits quantifier elimination (see section 5 below). In this paper, we now generalize that analysis to substructures of $\D$ induced by dense additive subgroups $G$ of $\R$ containing $\Z$, using the work of Robinson and Zakon \cite{robinson-zakon}. One can view these substructures in two ways :  as a quotient $G/\Z$ with operations induced by the structure $(G,\Z,+,-,0,1,<)$ (see for instance \cite{gunaydin},  chapter 8), or simply as  $\G:=(G\cap [0,\;1[, +_1,-_1, 0, <)$. We axiomatize the structures $\G$ and show that their theories are model-complete in a natural language (section 4) and decidable whenever the set of non-zero natural numbers such that $\G$ has no elements of order $n$ is recursively enumerable. We also give a direct proof that the theory of $(\R,\Z,+,-,0,1,<)$ is NIP,  being unaware of \cite{dolich-goodrick} at the time.  
\par Another approach, used by M. Giraudet, G. Leloup and F. Lucas, is through groups endowed with a cyclic order (see \cite{giraudet-leloup-lucas}, \cite{leloup-lucas}). 
Indeed, one can view $(G\cap [0,\;1[, +_1,-_1, 0, <)$ 
as a subgroup of the unit circle $(S^1,\cdot)$ with the induced circular order, namely the ternary relation $R(x,y,z)$ which holds whenever the points $x, y, z$ appear in that order when sending the interval $[0,1[$ to $S^1$ by the function $t\rightarrow e^{2\pi i t}$. An axiomatization of the structure $(S^1,\cdot, R)$ has been given by Lucas together with a quantifier elimination result \cite{lucas_hab}, \cite{lucas_ddg}. The two approaches  can be linked as follows. Let $R_{\mathbb D}$ be the ternary relation on $[0, 1[$ defined by $R_{\mathbb D}(x,y,z)  \leftrightarrow x<y <z \vee y<z<x \vee z< x<y$, and let $R_G$ be its restriction to $G\cap [0,\;1[$. For all $x, y\in [0,1[$, we have $x<y \leftrightarrow R_\mathbb D(0,x,y).$ Let $\mathbb G_c$ denote the structure $(G\cap [0,\;1[, +_1, 0, R_G)$, then $\mathbb G_c$ is a cyclically ordered group. Let $T_c$ be the Leloup-Lucas axiomatization of the first-order theory of $\mathbb G_c$ (see \cite[Definition 3.1, Theorem 4.12]{leloup-lucas}), and let $T$ be the axiomatization of the first-order theory of $\G$ given below (Definition 2.5). By the preceding remarks, the axioms of $T_c$ can be deduced from $T$ and vice versa, and various properties can be transferred between $T$ and $T_c$ in both directions. For example, model-completeness (see \cite[Proposition 4.9]{leloup-lucas}), the fact of not having the independence property (Corollary 4.3), decidability results (Corollary 4.3) etc.
Note that in order to study cyclically ordered groups, one important tool is to view them as quotient of a densely ordered group,  known as the {\it unwound}, by a cyclic group. 
\par Our approach is more elementary and direct, so we thought it useful to present it. 
The fact that we deal with an order instead of a cyclic order makes the approach possibly more intuitive. For instance it is easier to identify what is the largest convex (subsemi)group, or what are the archimedean models. Of course the counterpart is that we have to deal 
with two subsemigroups (the infinitely small elements and infinitely large ones), instead of 
 the group of infinitesimals only, in the cyclic order case.
\par In the case of $\mathbb D$, one of the key feature is that the torsion subgroup ($\Q\cap [0, 1[$) is dense. We decomposed an $\aleph_{1}$-saturated model into a standard part, the interval $[0,1[$,  and two semigroups, the positive infinitesimals and the infinitely large elements. Then, we used a Feferman-Vaught type argument which relied on the fact that  these subsemigroups are the positive (respectively the negative) part of  a divisible totally ordered group. 
\par Here we proceed as follows. We first write down a theory $T$ consisting of first-order properties of $\G$, making the distinction whether the subgroup of torsion elements of $\G$ is finite or not. As before, we decompose an $\aleph_{1}$-saturated model of $T$ into a divisible subgroup $\D_{\G}$ of $[0,  1[$ with the same torsion as $\G$, and a subgroup $H^{00}_{t}$ with the same indices as $\G$ (section 2). We show that $T$ is model-complete (section 4) using again a Feferman-Vaught type argument. The direct sum decomposition enables us to use that the theory of the subgroup $H^{00}$ of elements which are either infinitely small or infinitely large admits quantifier elimination, even though $H^{00}$ does not appear as a direct summand.  Let us note two special features. First, the standard part of an $\aleph_{1}$-saturated model is still the interval $[0,1[$ and so in general  its torsion is different from the torsion of $\G$. Second, the kernel of the standard part map, the subgroup $H^{00}$, does not coincide with the smallest type-$\emptyset$-definable subgroup of bounded index (which is known to exist since $T$ is NIP.) 

\par Set $\mathbb N^\ast=\mathbb N\setminus\{0\}$. We use boldface letters $\bf x$ to denote a tuple. For $H$ an abelian group, we denote by $H_{tor}$ the torsion subgroup of $H$. 
For basic facts on the model theory of abelian groups, see e.g. \cite{Z}.

 \section{Dense subgroups of $\D$}
 
Let $[0,1[ \; = \{r\in\mathbb R : 0\le r<1\}$ and $+_{1}$ denote addition modulo $1$. Then $([0,1[, +_1)$ is an abelian group isomorphic to the quotient of the additive group of real numbers by the additive group of the integers.   We add the order induced by the order on the real numbers and obtain the structure  $\mathbb D:=([0, 1[,+_1,-_1, 0,<)$. It is not an ordered group since the group operation $+_1$ is not necessarily compatible with the order (note that $\mathbb D$ has torsion elements). Recall that $\mathbb D$ is isomorphic to the circle group $(S^{1},\cdot,1)$ endowed with the order induced by the order on the unit interval through the exponential map $t\mapsto e^{2\pi it}.$ 
 The elements of order $n$, $n\geq 2$, are quantifier-free definable in $\D$. For instance, the element $\frac{1}{2}$ is definable by the formula $y\ne 0 \ \& \ 2y=0$ and in the same way, for $n>2$, the elements  $\frac{1}{n}$ are quantifier-free definable : 
$\frac{1}{n}$ is the (unique) $z$ such that $z\ne 0 \;\&\; nz=0 \; \& \; \wedge_{2\le j\le n-1} z< j z$. 

Let $G$ be a dense subgroup of $(\mathbb R, +,0)$. 
Robinson and Zakon  \cite{robinson-zakon} showed that such a structure can be  axiomatized as an ordered abelian group which is {\it regular dense}, namely for each positive integer $n$, $x<y$ implies that there is an element $z$ such that $x<nz<y$, and where for each $n\geq 2$, the indices
$[G:nG]$ (finite or infinite), are specified.

We will consider dense subgroups of $\mathbb D$. Each of them is induced by a dense subgroup $G$ of $(\mathbb R, +)$ containing $\mathbb Z$. Throughout the article we will fix such a $G$ and set the structure $\mathbb G = (G\cap [0,1[, +_1,-_1, 0, <)$, and $n_{\mathbb G}=[\mathbb G : n\mathbb G]$ ($n_{\mathbb G}\in\mathbb N$ or $n_{\mathbb G}=\infty$).

\lmm
For each $n\ge 2$, and for every $x,y\in \mathbb G$, $x<y$ implies that there exists  $z\in \mathbb G$ such that $x<nz<y$. 
Moreover, if $x<y<\frac{1}{n}$, then we have $(\bigwedge_{i=1}^{n-2} \exists z_{i}\;x<n.z_{i}<y\;\&\;\frac{i}{n}<z_{i}<\frac{i+1}{n})\;\&\;$\\$(\exists z_{0} \;(x<n.z_{0}<y\;\&\;z_{0}<\frac{1}{n}))\;\&\;(\exists z_{n-1}\;(x<n.z_{n-1}<y\;\&\;\frac{n-1}{n}<z_{n-1})).$
\elmm
\begin{proof} It follows from the fact that $\G$ is dense.
\end{proof}
\lmm 
For each $n\ge 2$, we have $[\mathbb G : n\mathbb G]\le [G:nG]\le n[\mathbb G : n\mathbb G],$ and if $\frac{1}{n}\in\mathbb G$ then $[\mathbb G : n\mathbb G] =[G:nG]$. In particular, if $\mathbb Q\subseteq G$, we get $[\mathbb G : n\mathbb G]= [G:nG]$.
\elmm
\begin{proof} The first inequality follows from the fact that $\Z\subset G$ and the second one from the fact that $[\Z:n\Z]=n.$
\end{proof}
\medskip
\par  The following observation is well-known.

\lmm  
 We have $\mathbb D_{tor} \subseteq \mathbb Q\cap [0,1[,$ and if $x\in \mathbb D$ has infinite order, then the cyclic subgroup $\mathbb Zx$ generated by $x$ in $\D$ is dense.
  \elmm
 \par We fix an enumeration $(c_n)_{n=0,1,2, \ldots}$ of $\mathbb G_{tor}$ with  $c_0=0$.
\par Let $L$ be the first-order language  $\{ +, -, 0, < \}$, and let $L_\rho=L\cup\{\rho_n: n\in \N^*\}$ be the expansion of $L$ by new constant symbols. The interpretation of the symbols $\rho_{n}$ in $\G$, denoted by $\rho_n^{\G}$,  is as follows.
\par {\bf Case $(I)$} When $\G_{tor}$ is dense, then $\rho_{n}^\G=c_{n}$, $n\in \N^*$. 
\par {\bf Case $(II)$} When $\mathbb G_{tor}$ is not dense in $\mathbb G$, we choose an element $\delta$ of infinite order 
and we set $\rho_{n}^\G=
n\delta$ if $n$ is odd, and $\rho_{n}^\G=-n\delta$ if $n$ is even. By the above observation $\Z\delta$ is dense in $\G$. 
\par In both cases, we will denote by $\rho_{\frac{i}{n}}$, $1\leq i\leq n-1$, the constant $\rho_{m}$ with $m$ minimal such that the interpretation $\rho_{m}^{\G}$ of $\rho_{m}$ in $\G$ belongs to the interval $]\frac{i-1}{n} , \frac{i}{n}].$
\medskip
\nota \label{N_G} Define $N_{\G}$ 
 as the set of $n\in \N^*$ such that $\G$ has no element of order $n$. Note that if $n\in N_{\G}$ then multiplication by $n$ in $\G$ is injective, and that $N_{\G}$ is a multiplicative subset of $\N$. So, if $N_{\G}\neq \emptyset$, then $N_{\G}$ is infinite and we have the subgroup $D_{N_{\G}}:=\{\frac{i}{n}:\;0\leq i\leq n-1,\;n\in N_{\G}\}$,   and we  denote by $\mathbb D_{N_{\G}}$ the corresponding  substructure of $\D$. 
 This subgroup of $\D$ is divisible (and so pure in $\D$), so it has a direct summand which we fix and denote by $\mathbb D_{\G}$ \cite[Theorem 13.3.1]{Hall}. Note that $\mathbb D_{\G}$  has the same torsion as $\G$.
\enota

 \defn\label{T}   
 
\par Let $T_{0}$ be the $L_\rho$-theory consisting of the following 
axioms $(1)-(5)$ and $(7)-(10)$. Let $T_{I}$ (respectively $T_{II}$) be the theory $T_{0}$ together with axiom $(6)_{I}$ (respectively axiom $(6)_{II}$). We will denote by $T$ the theory $T_{I}$ (respectively $T_{II}$) in case $\G_{tor}$ is dense (respectively in case $\G_{tor}$ is finite).
\begin{enumerate}
\item[(1)] the axioms of abelian groups;
\item[(2)] for each $n\in \N^*$,  the index of the subgroup of elements of the form $nx$ is equal to $n_\G$. 
\item[(3)] for each $n\in \N^*\setminus N_{\G}$, the axiom $$\exists\;z\; ((z\ne 0 \; \& \;nz=0 \; \& \; \wedge_{2\le j\le n-1} z< j z) \; \& \; \forall x(nx=0 \rightarrow (x=0 \vee \bigvee_{1\le j \le n-1} x=jz))$$ and for each $n\in  N_{\G}$, the axiom $\forall x\forall y (nx=ny \; \rightarrow \; x=y)$;
\item[(4)]  the relation $<$ is a strict order with minimum $0$ and without a maximum;
\item[(5)] $\forall x\forall y (x<y \rightarrow -y<-x)$; 
\item[$(6)_{I}$] in case $\G_{tor}$ is dense the axiom $$\exists z (\rho_n=iz \; \& \;  z\ne 0 \; \& \; mz=0 \; \& \; \wedge_{2\le j\le m-1} z< j z)$$ if $c_n=\frac{i}{m}$, $1\leq i\leq m$, $m,\;n\in \N^*$; 
\item[$(6)_{II}$] in case $\G_{tor}$ is finite, the axioms $n\rho_1\ne 0, n\in\N^*,$ and $\rho_n=n\rho_1$, if $n$ is odd, and $\rho_n=-n\rho_1$, if $n$ is even, and the axioms $$\exists z (z\ne 0 \; \& \;  mz=0 \; \& \; (\wedge_{1\le j\le m-1} z< j z) \, \& \, \rho_{n_1} < iz < \rho_{n_2})$$ if $c_n=\frac{i}{m}$ and $\rho_{n_1}^{\G}< c_n< \rho_{n_2}^{\G}$, $m\in \N^*$;
\item[(7)] for each $n\ge 2$, the axiom $x<y \rightarrow \exists z (x<nz<y)$ 
and the axiom 
\begin{eqnarray*}
 a<b<\rho_{\frac{1}{n}}&\rightarrow& \bigwedge_{i=1}^{n-2} \exists z_{i}\;\left(a<n.z_{i}<b\;\&\;\rho_{\frac{i}{n}}<z_{i}<\rho_{\frac{i+1}{n}}\right)\\ 
 && \;\&\;\exists z_{0} \;\left(a<n.z_{0}<b\;\&\;z_{0}<\rho_{\frac{1}{n}}\right)\\ 
 & & \;\&\; \exists z_{n-1}\;\left(a<n.z_{n-1}<b\;\&\;\rho_{\frac{n-1}{n}}<z_{n-1}\right);
 \end{eqnarray*} 

\item[(8)]  for all $n,m\in\mathbb N$, the axiom $\rho_n<\rho_m$, if $\rho_n^{\G}<\rho_m^{\G}$;
\item[(9)] for all natural numbers $m, n, m\prime, n\prime$ such that $\rho_{n\prime}^{\G}+\rho_{m\prime}^{\G}\le\rho_{n}^{\G}+\rho_{m}^{\G}<1$ or $1<\rho_{n\prime}^{\G}+\rho_{m\prime}^{\G}\le\rho_{n}^{\G}+\rho_{m}^{\G}$, (equivalently, $\rho_{n\prime}^{\G}+\rho_{m\prime}^{\G}\le\rho_{n}^{\G}+\rho_{m}^{\G}$ and $\rho_{n}^{\G}<-\rho_{m}^{\G}$ and $\rho_{n\prime}^{\G}+\rho_{m\prime}^{\G}\le\rho_{n}^{\G}+\rho_{m}^{\G}$
and $\rho_{n\prime}^{\G}>-\rho_{m\prime}^{\G}$)

the following axiom :
$$ \forall x\forall y
\left((\rho_{n\prime}\le x\le \rho_{n} \, \& \, \rho_{m\prime}\le y\le \rho_{m}) \rightarrow  \rho_{n\prime} + \rho_{m\prime} \le x+y\le \rho_{n} + \rho_{m}\right);$$ 
\item[(9)'] for all natural numbers $m, n, m\prime, n\prime$ such that $\rho_{n\prime}^{\G}=-\rho_{m\prime}^{\G}$ and $\rho_{n}^{\G}>-\rho_{m}^{\G}$, the following axiom :
$$ \forall x\forall y
\left((\rho_{n\prime}\le x\le \rho_{n} \, \& \, \rho_{m\prime}\le y\le \rho_{m}) \rightarrow x+y\le \rho_{n} + \rho_{m}\right);$$
\item[(10)] for each $n\ge 1$ such that $\rho_n^{\G}<\frac{1}{2}$, the axiom :
$$ \forall x\forall y\forall z \left[ (x\le \rho_{n} \ \& \ y\le\rho_{n} \ \& \ z\le \rho_{n} ) \rightarrow \left( (x\le y \leftrightarrow x+z\le y+z) \ \& \ (x\le x+y) \right)\right].$$
\end{enumerate}
\edefn

It is straightforward to verify that $\mathbb G$ is a model of $T$. It follows from axiom scheme (3) that all models have the same torsion as $\mathbb G$, and when there is $n$-torsion there are exactly $n$ elements of $n$-torsion. We will denote by $\frac{1}{n}$ the smallest such given by axiom 3, and accordingly $\frac{i}{n}$ will denote $i\frac{1}{n}$ in all models.
\lmm \label{consequencesT}
The following properties are consequences of  $T$.
\begin{enumerate}
\item[(a)] We have $(x<-x \, \& \, -y<y) \rightarrow x<y$.
\item[(b)] There exists $n$ such that $(x<y<\rho_n\rightarrow 2x<2y)$.
\item[(c)] We have $\rho_0=0$ and the set $\{\rho_n : n\in \mathbb N\}$ forms a subgroup.
\item[(d)] For each $n>2$, $(x<y<\rho_{1/n} \rightarrow nx<ny)$.
\item[(e)] For each $n\ge 2$, $(x<y<\rho_{1/n} \rightarrow \exists z_{0} \;(x<n.z_{0}<y\;\&\;z_{0}<y))$.
\end{enumerate}
\elmm
\begin{proof} Item (a) : assume that $x<-x, -y<y$ and $y\le x$; so $-x\le -y$ by axiom (4), which implies that $x<-x\le -y<y$, a contradiction; therefore $(x<-x \, \& \, -y<y) \rightarrow x<y$. 
\par Item (b) follows from axiom (10). 
Item (c) follows from  axioms (6)$_I$, (6)$_{II}$. 
 Item (d) follows from axioms (9), (10). 
 Item (e) follows from axioms (7), (9), (10). 
\end{proof}

\defn  Let $H$ be a model of $T$, we set 
$H^{00}_{+}= \{x\in H: x<\rho_{n}, \forall n\in\mathbb N^\ast \}$,  $H^{00}_{-}=\{x\in H: -x<\rho_{n}, \forall n\in\mathbb N^\ast\}$ and $H^{00}=H^{00}_{+}\cup H^{00}_{-}$. Note that both $H^{00}_{+},\;H^{00}_{-}$ are $\bigdoublewedge$-$\emptyset$-definable, i.e. a countable intersection of definable subsets of $H$ without parameters.
\edefn
 In section 4, we will show that $T$ is NIP. So, we know that $H$ has a smallest type-definable subgroup of bounded index \cite{Sh}. A natural candidate is the subgroup $K:=\bigcap_{\{n\in\mathbb N: n_{\G} \in \N\}} n H^{00}$, which is, in general, a proper type-definable subgroup of $H^{00}$ of bounded index. 
 We did choose that notation because in case $\G=\D$, $K=H^{00}$ and it is well-known that $H^{00}$ is the smallest type definable of bounded index.

\lmm \label{Tmo} 
Let $H$ be a model of  $T$, and $a,b\in H$ such that $a,b\in H^{00}_+$ and  $a<b$, then $b-a \in H^{00}_+$.  
\elmm 
\begin{proof}Set $c=b-a$ and assume that $c\not\in H^{00}_+$.  Then for some $n\ge 1$ we have $\rho_n\le c$. 
If $c\not\in H^{00}_{-}$, then there exists $n_1\ge 1$ such that $\rho_{n_1}\le -c$, therefore $c\le -\rho_{n_1}=\rho_{n_2}$ (Lemma \ref{consequencesT}). By axiom (9), we obtain $\rho_n\le a+c=b$, a contradiction. 
If $c\in H^{00}_{-},$ then $-c\in H^{00}_{+}$ and then by axiom (10), we get $b\le b-c=a$, a contradiction. \end{proof}

\lmm 
Let $H$ be a model of $T$. Then $H^{00}$ is a torsion-free subgroup of $H$. 
 \elmm
 \begin{proof} By definition we have  $0\in H^{00}_+, x\in H^{00}_+ \rightarrow -x\in H^{00}_-, x\in H^{00}_- \rightarrow -x\in H^{00}_+$. The set $H^{00}_+$ is stable by $+$ (axiom (9)), and so is $H^{00}_-$ using the function $-$. Let $x\in H^{00}_+, y\in H^{00}_-$, non-zero, and consider $x+y$. If $x=-y$, then $x+y=0$. If $x>-y$, then $x-(-y)=x+y\in H^{00}_+$ by Lemma \ref{Tmo}. If $x<-y$, then $-y-x\in H^{00}_+$ by Lemma \ref{Tmo}, so $x+y\in H^{00}_-$. Thus $H^{00}$ is  a subgroup. 
\par In case (I), $H^{00}$ is torsion-free since $\{\rho_n:\;n\in \N^*\}$ is the set of the torsion elements. In case (II), if follows from axioms $(7)$ and $(9)$.
 \end{proof}

\par With the induced order $H^{00}$ is not an abelian ordered group since $0$ is a minimum, unless it is the trivial group. However  we will check that $H^{00}_+$ is the positive part of a unique abelian totally ordered group ($H^{00}_-\setminus\{0\}$ its negative part). This follows from well-known results \cite{C}, which we will recall in the next section.

\begin{definition} Let $H$ be an $\aleph_{1}$-saturated model of $T$ and let 
  $\frac{i}{\ell}\in \D_{tor}$ with $\ell\in N_{\G}$, $1\le i<\ell$. We define $(\frac{i}{\ell}):=\{h\in H: H\models \rho_{n}<h<\rho_m, \mbox{ whenever }\D\models \rho_n<\frac{i}{\ell}<\rho_m, m,n\in\mathbb N^\ast\}$
  and call it the {\it cut} in $H$ determined by $\frac{i}{\ell}$. \end{definition} 

 \par Note that one can add two such cuts in $H$, namely $(\frac{i}{n})+(\frac{j}{m})$ and get the cut $(\frac{im+jn}{n.m})$, with $n,\;m\in N_{\G}$, $1\leq i<n,\,1\leq j<m$. Since $N_{\G}$ is a multiplicative set, this will follow from the next lemma.

\lmm\label{coupures} Let $H$ be an $\aleph_{1}$-saturated model of $T$, $n\in N_{\G}$ and let $z\in H$ such that $z\in (\frac{1}{n})$. Then $nz\in H^{00}$ and $-z\in (\frac{n-1}{n})$. Moreover if $nz\in H^{00}_{-}$, then  $-nz\in H^{00}_{+}$.
\elmm
\begin{proof} Let $z\in (\frac{1}{z})$. Note that $-\frac{1}{n}=\frac{n-1}{n}$, and $z\ne \rho_m$, for all $m$. First we check that $-z\in (\frac{n-1}{n})$. Suppose $\rho_{k_1}<\frac{n-1}{n}<\rho_{k_2}$, we have to see that $\rho_{k_1}<-z<\rho_{k_2}$. We have $\rho_{k_1}<-\frac{1}{n}<\rho_{k_2}$. By Axiom 5, we get $-\rho_{k_2}<\frac{1}{n}<-\rho_{k_1}$. Hence $\rho'_{k_2}<\frac{1}{n}<\rho'_{k_1}$, where $\rho'_{k_1}=-\rho_{k_1}, \rho'_{k_2}=-\rho_{k_2}$ (lemma 2.6(c)). Since $z\in(\frac{1}{n})$, we get $\rho'_{k_2}<z<\rho'_{k_1}$. Using Axiom 5 again we get $\rho_{k_1}<-z<\rho_{k_2}$, as wanted. Now we check that $nz\in H^{00}$. Suppose not, then we get some inequality of the form $\rho_{m_1}<nz<\rho_{m_2}$. Since the subgroup generated by $n\rho_{1}^\G$ is also dense in $[0, 1[$, we may assume that $\rho_{m_1}=n.\rho_{k_1}, \rho_{m_2}=n.\rho_{k_2}$, where $\rho_{k_1}, \rho_{k_2}<\frac{1}{n}.$ We get $n\rho_{k_1}<nz<n\rho_{k_2}$. Now we have $z\ne \rho_{k_1}, \rho_{k_2}$, so by Axiom 9 we must have $ \rho_{k_1}<z< \rho_{k_2}$, which contradicts $z\in(\frac{1}{n})$. Finally, the last assertion of the lemma follows directly from Axiom (5).
\end{proof}

 \par Let $H$ be a model of $T$, we set $H_\rho=\{\rho_n^H : n\in \mathbb N\}$. In the following definition, dually to Definition 2.10, we define the cuts in $H_{\rho}$ determined by the elements of $H$,
 and we define the standard part map using the density of $\{\rho_n^{\G}: n\in \mathbb N\}$ in $\D$.
 
\defn\label{standard}  Let $H$ be a model of $T$. 
Any  $g\in H$ determines a {\it cut} in $H_\rho$, namely the following pair of subsets of $H_{\rho}$:  
$C^-(g):=\{h\in H_{\rho}:\;h<g\}$,  $C^+(g):=\{h\in H_{\rho}:\;g<h\}.$ We note that $C^+(g)=\emptyset \leftrightarrow g\in H^{00}_{-}$ and $C^-(g)=\emptyset \leftrightarrow g\in H^{00}_{+}$. We define the {\it standard part map} $st : H\to \mathbb D$ as follows.
We identify $\rho_n$ with the corresponding real number, and we set $st(\rho^H_n)=\rho_n^{\G}$, and for $g\not\in H_{\rho}$, we let $st(g)$ be the real number $r$ such that  $C^-(g)\le r\le C^+(g)$ if $C^-(g), C^+(g)\ne \emptyset$, and $st(g)=0$ otherwise. 
\edefn
\lmm \label{partie_standard} Let $H$ be an $\aleph_{1}$-saturated model of $T$. Then the map $st : H\to \mathbb D$ is a morphism of abelian groups and its kernel is equal to $H^{00}$. If $H$ is $\aleph_{1}$-saturated, then $H/H^{00}\cong \D$.
 \elmm
 \begin{proof}It is straightforward that $st(-x)=-st(x)$ and $st(x)=0 \leftrightarrow x\in H^{00}$. Let us show that $st(x+y)=st(x)+st(y)$. \par If $x, y\in H^{00}$, then $x+y\in H^{00}$ and so we are done. 
\par If $x, y\not\in H^{00}$, then we have $C^-(x), C^+(x), C^-(y), C^+(y)\ne \emptyset$ and we add the cuts as in $\mathbb R$ using axiom (9)  except when
$x+y=0$. In the case where $x\in (\frac{i}{n})$ and $y\in (\frac{n-i}{n})$ with $1\leq i\leq n-1$, we use Lemma \ref{coupures}. In the other cases, we use the fact that $\Q$ is dense and so for any $\rho_n\in C^{\pm}(x)\cap (q_1)$, $\rho_m\in C^{\pm}(y)\cap (q_2)$ we have that $\rho_n+\rho_m\in (q_1+q_2)$.
\par If $x\in H^{00}_+$ and $y\not\in H^{00}$, then by axiom (9) we have $C^+(x+y)=C^+(y)$ and $C^-(x+y)=C^-(y)$, so $st(x+y)=st(y)$. 
\par If $x\in H^{00}_-$ and $y\not\in H^{00}$, we reduce ourselves to the preceding case using the function $-$. 
\par If $H$ is $\aleph_{1}$-saturated, then all the cuts are realized. \end{proof}

\section{Monoids and regular groups}
\par We want to axiomatize the theory of $H^{00}$. We introduce the following schemes of axioms. 
\defn\label{ax-reg-ind} 
Recall that for $n\ge 2$, $n_{\G}=[\mathbb G:n\mathbb G].$ Consider the three axiom schemes:
\begin{enumerate}
\item[(1)] (regularity) For each $n\ge 2$, the axiom : $x<y \rightarrow \exists z (x<nz<y)$.
\item[(2)] (indices) For each $n\ge 2$, the axiom :
$$ \exists x_1,\ldots, x_{\tilde n_{\mathbb G}} \left( \bigwedge_i \neg \exists y (x_i=ny) \; \& \;  \bigwedge_{i<j} \neg \exists z_1,z_2 (x_i+nz_1=x_j+nz_2) \; \& \; \forall x\exists z \bigvee_i x=x_i+nz\right)$$
where $\tilde n_\mathbb G=n_{\G}$ if $n\notin N_{\G}$, and $\tilde n_\mathbb G=n_{\G}+n-1$ otherwise, with the convention that if $n_{\mathbb G}=\infty,$ then we have the corresponding infinite list of axioms.
\item[(3)]
 (indices in $\G$) For each $n\ge 2$, the axiom : 
$$ \exists x_1,\ldots, x_{n_{\mathbb G}} \left( \bigwedge_i \neg \exists y (x_i=ny) \; \& \;  \bigwedge_{i<j} \neg \exists z_1,z_2 (x_i+nz_1=x_j+nz_2) \; \& \; \forall x\exists z \bigvee_i x=x_i+nz\right).$$
\end{enumerate}
\edefn

\defn Let $L_{mo}$ be the language $\{+, 0, <\}$ 
and $T_{mo}$ be the $L_{mo}$-theory consisting of the following axioms :  
the axioms for commutative monoids, the relation $<$ is a total order, $\forall a\forall b\forall c \; (a+c\leq b+c\leftrightarrow a\leq b)$. Let $T_{mon}$ be $T_{mo}$ together with the following axioms : $\forall a\forall b (a\le a+b)$, $\forall a\forall b\; ( a<b\rightarrow \exists c(a+c=b))$. 
\edefn

\defn\label{ind} 
 Let $T_{mr00}$ be the $L_{mo}$-theory extending 
$T_{mon}$ by adding the schemes of axioms $(1),\;(2)$ above (Definition \ref{ax-reg-ind}).
 Let $T_{mr\mathbb G}$ be the $L_{mo}$-theory 
extending $T_{mon}$ by adding the schemes of axioms $(1),\;(3)$ above (Definition \ref{ax-reg-ind}).
\edefn

\lmm  \label{TmrG}
Let $H\models T$, then $H^{00}_+\models T_{mr00}$. 
\elmm
\begin{proof} By Lemma \ref{Tmo}, $H^{00}_+$ is a model of $T_{mon}$. Since $H$ is a model of $T$,  $H^{00}_{+}$ 
satisfies axiom $(1)$ by Lemma \ref{consequencesT}.  Assume that $H$ is $\aleph_1$-saturated. The standard model $\mathbb G$ implies the consistency of the following type $tp(x_1,\ldots,x_{n_{\G}})$ :
$$ \left\{   \neg \exists y (x_i=ny) ,  \neg \exists z_1\exists z_2 (x_i+nz_1=x_j+nz_2) , x_i< \rho_m : i, j=1, \ldots, n_{\G}, i\ne j, m\in \mathbb N\right\}$$
So it is realised in $H$, by, say, $h_1, \ldots, h_{n_{\G}}$. We have $h_1, \ldots, h_{n_{\G}}\in H^{00}_+$ and they belong to different cosets modulo $nH$ and so $[H^{00}:nH^{00}]\ge n_{\G}$ and  $[H^{00}_+:nH^{00}_+]\ge n_{\G}$.
\par Let $n\in \N\setminus(N_{\G}\cup\{0\}$), so we have $\frac{1}{n}\in H$. Suppose that $n.y\in H^{00}$ and so for some $i$, $y\in(\frac{i}{n})$, $1\leq i\leq n-1$. We get $y-\frac{i}{n}\in H^{00}$ and so $[H^{00}:nH^{00}]= n_{\G}$ and $[H^{00}_+:nH^{00}_+]= n_{\G}$.
\par Let $n\in N_{\G}$. Since $H$ is $\aleph_{1}$-saturated, there is at least one element $z_{i}$ in the cut $(\frac{i}{n})$, $1\leq i\leq n-1$ and $nz_{i}\in H^{00}_{+}\setminus \{0\}$ (Lemma \ref{coupures}). Note that this element is not in the image of another element of $H$ under the map $x\mapsto nx$ (this would create $n$-torsion). In particular $nz_{i}$ belongs to a new coset of $nH^{00}$. The number of new cosets is equal to $n-1$ and since we specify the torsion in the standard model (axiom (3)), the number of these new cosets is the same in every model. We obtain $[H^{00}:nH^{00}]=(n-1)+n_{\G}$, whence $[H^{00}_+:nH^{00}_+]=(n-1)+n_{\G}$. \end{proof}

\defn
Let $H$ be an $\aleph_1$-saturated model of $T$. For each $n\in N_{\G}$ and $1\leq i\leq n-1$, we denote by $(\frac{i}{n})^+$ (respectively $(\frac{i}{n})^-$) the set of all elements $z$ 
 which are in the cut of $\frac{i}{n}$ and such that $n.z\in H^{00}_{+}$ (respectively $n.z_{i}\in H^{00}_{-}$). 
 We consider the monoid generated by $H^{00}_{+}$ (respectively $H^{00}_{-}$) and all the $(\frac{i}{n})^+$ (respectively $(\frac{i}{n})^-$).
 We denote this monoid  by $H^{00}_{+,t}$ (respectively $H^{00}_{-,t}$) and we set $H^{00}_{t}:=H^{00}_{+,t}\cup H^{00}_{-,t}$.
 \edefn
 \par Note that $H^{00}_{+,t}$=$H^{00}_{+}\cup\{(\frac{i}{n})^+;\;n\in N_{\G},\;1\leq i\leq n-1\}.$

\nota\label{cosets}
 We denote by $u_{1},\cdots,u_{n_{\G}}$ the cosets representatives of $nH_{+}^{00}$ inside $H_{+}^{00}$ which do not belong to $nH\setminus nH_{+}^{00}$ (note that $n_{\G}$ may be infinite and in this case the enumeration of the coset representatives is infinite) and by $v_{i,n}$, $1\leq i\leq n-1$, the coset representatives which are in the image of the cut $(\frac{i}{n})^+$ by the map $x\rightarrow nx$. We will make the convention that $u_{1}\in nH_{+}^{00}$ and for convenience we will also denote that element by $v_{0,n}$.
\enota
\lmm Let $H$ be an $\aleph_1$-saturated model of $T$. Then for every $n\in \N^\ast$, the indices of 
$nH^{00}_{+,t}$ in $H^{00}_{+,t}$ are the same as those for $\G$. 
Moreover $H^{00}_{+,t}$ satisfies the regularity axiom scheme and is a model of $T_{mr\mathbb G}.$
\elmm
\par
\begin{proof}We first show that $H^{00}_{+,t}$ is pure in $H$, which implies for $n\in \N^\ast$ that $[H^{00}_{+,t}:nH^{00}_{+,t}]= n_{\G}.$
 \par First assume that $ny\in H^{00}_{+}$ and $y\notin H^{00}_{+}$.   Either $n\notin N_{\G}$, so $y$ is of the form $\frac{i}{n}+z$, for some $z\in H^{00}_{+}$ and $1\leq i\leq n$. Therefore we found an element in $H^{00}_{+}$ such that $ny=nz$. Or $n\in N_{\G}$, so $y\in (\frac{i}{n})^+$, for some $1\leq i\leq n$, and so $y\in H^{00}_{+,t}$.
 \par Second, assume that $ny\in (\frac{j}{m})^+$ with $m\in N_{\G}$, $1\leq j<m$, then $y$ is in a cut of the  form $\frac{k}{nm}$,  $nm\in N_{\G}$, $1\leq k<nm$. So again $y\in H^{00}_{+,t}$.

\par Now let us show that  $H^{00}_{+,t}$ is regular. Let $x<y$ in $H^{00}_{+,t}$. We distinguish the following cases.
\par If $x, y\in H^{00}_{+}$, then we apply axiom (7) of Definition \ref{T} and we find an element $z_{0}\in H^{00}_{+}$ such that \mbox{$x<n.z_{0}<y$}.
\par If $x, y\in (\frac{i}{m})^+$, for $m\in N_{\G}$ and $1\leq i<m$, we apply the regularity of $H$ (axiom (7) of Definition \ref{T}) and we find an element $h\in H$ such that $x<nh<y$. But $N_{\G}$ is a multiplicative set, so $h\in(\frac{j}{nm})^+$ which belongs to $H^{00}_{+,t}$, $1\leq j<nm$.
\par Finally, if $x,\;y$ belong to different cuts, we use the fact that the quotient $H^{00}_{t}/H^{00}\cong \D_{N_{\G}}$ is dense in $\D$.   
\end{proof}

 \lmm Let $H$ be an $\aleph_1$-saturated model of $T$. Then $H^{00}_{+,t}$ is also $\aleph_{1}$-saturated as a pure monoid.
 \elmm 
 \begin{proof}We have to show that any system of positive primitive formulas with one free variable and with parameters in $H^{00}_{+,t}$  which is finitely satisfiable is satisfiable. It is straightforward  that $H^{00}$ (respectively $H^{00}_{+}$) is $\aleph_1$-saturated and so if this system is finitely satisfiable by an element of $H_{+}^{00}$, it is immediate. Otherwise it is finitely satisfiable by an element of  $z_{i,n}+H^{00}_{+}$ for some $z_{i,n}\in (\frac{i}{n})^+$, $n\in N_{\G}$. Since all the cosets $z_{i,n}+H^{00}$ are disjoint, we may assume we stay in the same coset. We may then apply the $\aleph_{1}$-saturation of $H$. \end{proof} 
\bigskip

\par We decompose the problem in two parts: a divisible subgroup of $([0,1[,+,0)$ with the same torsion as $\G$ on one hand, and on the other hand the subgroup $H^{00}_{t}$ which has the same indices as $\G$ but which is torsion-free. To a nonzero element $u$ of $H^{00}_{t}$, we associate a couple $(u_{+},\frac{i}{n})$, $u_{+}\in H_{+}^{00}\setminus \{0\}$, if $u\in (\frac{i}{n})^+$, $1\leq i\leq n-1$, or $(u,0)$ if $u\in H_{+}^{00}$, or  $(u_{-},\frac{i}{n})$, $u_{-}\in H_{-}^{00}\setminus\{0\}$, if $u\in (\frac{i}{n})^-$, $1\leq i\leq n-1$, or $(u,0)$ if $u\in H_{-}^{00}$. The set of these couples is endowed with the lexicographic order.
\par We have  $H^{00}_{t}/H^{00}\cong \D_{N_{\G}}$. 
The first step consists in studying the group $H^{00}=H^{00}_+\cup H^{00}_-$. 
\medskip
\par Let $S_1,\;S_2$ be two commutative monoids with $S_1\subseteq S_2$. We say that $S_1$ is pure in $S_2$ if for all $a\in S_1$ whenever there exists $b\in S_2$ such that $n.b=a$, then there exists $c\in S_1$ such that $n.c=a$, $n\in \N$. 
\lmm \label{modelecompleteTmrG}
 Let $S_1, S_2$ be two models of $T_{mr00}$ such that $S_1\subseteq S_2$ and assume that $S_1$ is a pure submonoid of $S_2$. Then $S_1$ is an elementary substructure of $S_2$.
\elmm
 \begin{proof}Let $S_1, S_2$ be two models of $T_{mr00}$ such that $S_1\subseteq S_2$. By \cite{C}, let $G_1, G_2$ be ordered abelian groups such that $S_k$ is the positive part of $G_k, k=1,2$. 
 The inclusion $S_1\subseteq S_2$ induces an inclusion of $G_1$ in $G_2$, representing $x\in G_1$ as $x=s-s', s, s'\in S_1$. We get that $G_1$ is a substructure of $G_2$. We claim that $G_k, k=1,2,$ is regular dense and $[G_k:nG_k]=n_{\mathbb G}$, for each $n\ge 2$. Then it follows by results of  Robinson-Zakon (\cite{robinson-zakon}) that $G_1\subseteq_{ec} G_2$, whenever $G_1$ is a pure subgroup of $G_2$, thus  $S_1\subseteq_{ec} S_2$, as wanted. We check the properties of $G_k$. Let $x,y\in G_k$ such that $x<y$ and $n\in \N$, $n\ge 2$. If $x,y\in S_k$, then there exists $z\in S_k$ such that $x<nz<y$, as wanted. If $-x, -y\in S_k$, then  we have $-y<-x$ and there exists $z\in S_k$  such that $-y<nz<-x$, so $x< n(-z) <y$, as wanted. If $-x, y\in S_k, y\ne 0$, then $0<y$ and there exists $z\in S_k$ such that $0<nz<y$, so $x<nz<y$, as wanted. This shows that $G_k$ is regular dense. Let $n\ge 2$, we now check that $[G_k:nG_k]=\tilde n_{\mathbb G}$. Let $x_1, \ldots, x_{n_{\mathbb G}}\in S_k$ as in axiom (2) of Definition \ref{ax-reg-ind}. Let $x\in G_k$. If $x\in S_k$, then $x\in x_i+nS_k\subseteq x_i+nG_k$ for some $i$. If  $ -x\in S_k,$ then  $-x\in x_i+nS_k$ for some $i$, say $-x=x_i+ny, y\in S_k$, so  $x=-x_i+n(-y)=(n-1)x_i+nx_i+n(-y)=x_j+nz+n(x-y)=x_j+n(z+x-y)$, for some $j$ and $z\in S_k$. This shows that $[G_k:nG_k]\le n_{\mathbb G}$. We claim that $x_i-x_j\not\in nG_k, i\ne j$, and this implies $[G_k:nG_k]\ge n_{\mathbb G}$. Indeed, if we had $x_i-x_j\in nG_k, i\ne j$, say $x_i-x_j=nx, x\in G_k$, we would get $ z_1, z_2\in S_k$ such that $x=z_2-z_1$, and $x_i+nz_1=x_j+nz_2$ which contradicts axiom (2) of Definition \ref{ax-reg-ind}. \end{proof}

\nota Let $L_{mo}'$ be the language $L_{mo}$ expanded, for each $n\ge 2$, by the binary predicate $D_n(x,y)$ defined by $D_n(x,y)\leftrightarrow \exists z_1\exists z_2 (x+nz_1=y+nz_2)$.
\enota
 \lmm \label{eq-mrG}
 The theory $T_{mr00}$ admits quantifier elimination in the language $L_{mo}'$.
 \elmm
 \begin{proof} We apply the following  criterion for quantifier elimination (see for instance \cite{marker}, 3.1.6) :
let $S_1,S_2$ be models of $T_{mr\mathbb G}$ and  $A$ such that $A\subseteq S_1, A\subseteq S_2$ as an $L_{mo}'$-substructure, let $\phi({\bf y}, x)$ be a quantifier-free $L_{mo}'$-formula  and ${\bf a}\in A$, and  assume there exists $b\in S_1$ such that $S_1\models \phi({\bf a}, b)$, then there  exists $c\in S_2$ such that $S_2\models \phi({\bf a}, c)$. Indeed, by \cite{C}, let $G_i$ be the abelian totally ordered group such that $S_i$ is the positive part of $G_i, i=1,2$. Let $G_A$ be the subgroup generated by $A$ such that $G_A\subseteq G_1, G_A\subseteq G_2$.  Then $G_i$ is a regular dense abelian totally ordered group. Let us check that $G_A$ is a common $L_{mo}'$-substructure of $G_1, G_2$. We need to check that given $x, y\in G_A$, $G_1\models D_n(x,y)  \leftrightarrow G_2\models D_n(x,y)$. Write $x=d-e, y=f-g$, with $d,e,f,g\in A$. Assume that, say, $G_1\models D_n(x,y)$. Then there exist $z_1, z_2\in G_1$ such that $x+nz_1=y+nz_2$. We may suppose that $z_1, z_2\in S_1$. We obtain $d-e+nz_1=f-g+nz_2$, so $d+g+nz_1=f+e+nz_2$, and then  $S_1\models D_n(d+g, f+e)$. So $S_2\models D_n(d+g, f+e)$, since $A$ is a common $L_{mo}'$-substructure of $S_1, S_2$. We obtain $G_2\models D_n(x,y)$. 
\par We are now in the position of applying the quantifier elimination result of V. Weispfenning for ordered regular dense abelian groups in the language of ordered abelian groups with the predicates  $D_n,$ $n\ge 2$ (\cite{weispfenning1981}). We obtain $G_1\models \exists u(u\ge 0 \& \phi({\bf a}, u))$, which implies that $G_2\models \exists u (u\ge 0 \& \phi({\bf a}, u))$, and any such element $u$ gives the sought after element $c$. \end{proof}
 
 \defn
 Let $T_{r\mathbb G, <}$ be the $L$-theory of torsion-free abelian groups together with the following axioms : the relation $<$ is a strict total order; the function $x\mapsto -x$ swaps the order; the subset $\{ x : x<-x\} \cup \{0\}$ is a  model of $T_{mr00}$; and finally the axiom $(x<-x \; \& \; -y<y )\rightarrow x<y$.
 \edefn

 \lmm  
 Let $H$ be an $\aleph_1$-saturated model of $T$, then $H^{00}$ is a model of $T_{r\mathbb G, <}$.
 \elmm
 \begin{proof} We have already seen that $H^{00}_+ = \{ x : x<-x\} \cup \{0\}$. The result follows from Lemmas \ref{TmrG} and \ref{consequencesT}. \end{proof}

 \lmm  \label{modelecompleteTrg<}
 Let $F_{0}, F_{1}$ be two models of $T_{r\mathbb G,<}$ such that $F_{0}\subseteq F_{1}$ and assume that $F_1$ is a pure subgroup of $F_2$. Then $F_1$ is an elementary substructure of $F_2$.
 \elmm
 \begin{proof} 
For a model $F$ of $T_{r\mathbb G,<}$, let $F_{+}:=\{x\in F: x<-x\}\cup \{0\}$, then $F_{+}$ is a model of $T_{mr00}$. Let $F_{0}, F_{1}$ be two models of $T_{r\mathbb G,<}$ such that $F_{0}\subseteq F_{1}$. Then $F_{0,+}\subseteq F_{1,+}$ and by Lemma \ref{modelecompleteTmrG} we have $F_{0,+}\subseteq_{ec} F_{1,+}$.
By the compactness theorem, there exists 
a model $F_2$ of $T_{r\mathbb G,<}$ such that $F_0\subseteq_{ec} F_2$ and $F_{0,+}\subseteq F_{1,+}\subseteq F_{2,+}$. 
  Let $i:F_{1,+}\to F_{2,+}$ be the inclusion. Note that since the order $<$ is total, we have either $x\in  F_{1,+}$ or $-x\in  F_{1,+}$, but not both if $x\ne 0$. Let $i^\ast : F_1\to  F_2$ defined by $i^*(x)=x$ if $x\in F_{1,+}$ and $i^*(x)=-i(-x)$ if $-x\in F_{1,+}$. 
We check that $i^\ast$ is a morphism and so $F_1$ is a substructure of $F_2$. Indeed, we have $i^\ast(0)=0$, and so $i^\ast(-x)=-i^\ast(x)$. To check that $i^\ast(x+y)=i^\ast(x)+i^\ast(y)$, it suffices to consider the case where $x\in  F_{1,+}$ and $ -y\in  F_{1,+}$. If $x+y\in F_{1,+}$, we have $i^\ast(x)=i^\ast(x+y+(-y))=i(x+y+(-y))=i(x+y)+i(-y)=i^\ast(x+y)+i^\ast(-y)$\\$=i^\ast(x+y)-i^\ast(y)$, therefore $i^\ast(x+y)=i^\ast(x)+i^\ast(y)$. If $-(x+y)\in  F_{1,+}$, then using $-(x+y)=(-x)+(-y)$ we get back to the preceding case. It remains to verify that $x<y$ implies that  $i^\ast(x)<i^\ast(y)$. If $x,y \in F_{1,+}$, it is straightforward. If $-x, -y\in F_{1,+}$, we use the fact that $-$ reverses the order. The only other possibility is that $x, -y\in F_{1,+}$. In this case we obtain $x, -y\in F_{2,+}$, namely $i^\ast(x), -i^\ast(y)\in F_{2,+}$, so $i^\ast(x)<i^\ast(y)$ since $F_{2}\models T_{r\mathbb G,<}.$ We obtain $F_0\subseteq F_1\subseteq F_2$, and since $F_0\subseteq_{ec} F_2$ we get  $F_0\subseteq_{ec} F_1$. \end{proof}
 
  \lmm  \label{eqTrg<}
  Let $L':=L\cup \{D_n(x,y): n\in \N^*\}$. The theory $T_{r\mathbb G, <}$ admits quantifier elimination in $L'$ and is complete.
  \elmm
 \begin{proof} Given a model $F$ of $T_{r\mathbb G, <}$, set $F_{+}:=\{x\in F: x<-x\}\cup \{0\}$. Then $F_{+}\models T_{mr\mathbb G}$. 
 \par We use the same criterion as above to prove quantifier elimination : let $F_1,F_2$ be models of $T_{r\mathbb G, <}$ and $A$ such that  $A\subseteq F_1, A\subseteq F_2$ as an $L'$-substructure, let $\phi({\bf y}, x)$ be a quantifier-free $L'$-formula and let ${\bf a}\in A$., Suppose that there exists $b\in F_1$ such that $F_1\models \phi({\bf a}, b)$, then there exists $c\in F_2$ such that $F_2\models \phi({\bf a}, c)$. Let $A_+ =\{ x\in A : x<-x\} \cup \{ 0\}$. Then $A_+$ is a common $L'$-substructure of $F_{1,+}, F_{2,+}$. By Lemma \ref{eq-mrG}, we have $(F_{1,+}, A_+)\equiv_{L_{mo}'} (F_{2,+}, A_+)$. Since either $z\in F_{k,+}$ or $-z\in F_{k,+},$  we obtain for all $x,y\in F_{k,+},$ that $F_k\models D_n(x,y)$ iff there exist $z_1, z_2\in F_{k,+}$ such that $x+nz_1=y+nz_2.$ So we may consider $F_{k,+}$ as an $L_{mo}'$-substructure of $F_k$. By the compactness theorem, there exists a model $F$ of $T_{r\mathbb G, <}$ such that $F_2\preceq F$ and $F_{1,+}\subseteq F$ as $L_{mo}'$-substructures. Let $i:F_{1,+}\to F$ denote the inclusion map, and let $i^\ast : F_1\to F$ be the map defined by $i^\ast(x)=x$ if $x\in F_{1,+}$, and $i^\ast(x)=-i(-x)$ if $-x\in F_{1,+}.$ As in Lemma \ref{modelecompleteTrg<}, we have $i^\ast(0)=0, i^\ast(-x) =-i^\ast(x), i^\ast(x+y)=i^\ast(x)+i^\ast(y),$ and $i^\ast$ is injective. Moreover $F_1\models D_n(x,y)$ iff $F\models D_n(i^\ast(x),i^\ast(y))$. The forward direction is clear. For  the converse, assume that $F\models D_n(i^\ast(x),i^\ast(y))$. We distinguish the following cases. If $x,y\in F_{1,+}$, then we immediately obtain that $F_{1,+}\models D_n(x,y)$, and so $F_1\models D_n(x,y)$. If $-x, -y\in F_{1,+}$, then $F_1\models D_n(-x,-y)$ by the preceding case, so $F_1\models D_n(x,y)$. If $-x, y\in F_{1,+}$, say $u_1, u_2\in F$ such that $-i(-x)+nu_1=y+nu_2$, then $nu_1=i(-x)+y+nu_2$. Therefore $F\models D_n(i^\ast(0), i^\ast (i(-x)+y))$, and by the first case, $F_1\models D_n(0, -x+y)$, so $F_1\models D_n(x,y)$. Then, $i^\ast$ embeds $F_1$ into $F$ as an $L'$-substructure. We obtain $F\models \exists u (\phi({\bf a}, u)$, so $F_2\models \exists u (\phi({\bf a}, u)$.  
 \par Let us show that $T_{r\mathbb G, <}$ is complete. Let $\si$ be an $L'$-sentence. Since $T_{r\mathbb G,<}$ admits quantifier elimination, we may assume that $\sigma$ is quantifier-free and so it it is a boolean combination of sentences of the form $D_n(0,0)$ or $0=0$. 
  \end{proof}

\section{Model-completeness and completeness}
\par  Let $L_\rho':=L_\rho\cup \{D_n(x,y): n\in \N^*\}$. 
\thm \label{eqL} 
The $L_\rho'$-theory $T$ is model-complete.
\ethm

Recall that we have fixed the following direct sum decomposition: $\D=\D_{N_{\G}}\oplus  \D_{\G}$ (\ref{N_G}), with $\mathbb D_{\G}$  having the same torsion as $\G$.

\lmm \label{proj_ss} Let $M\subseteq N$ be two $\aleph_1$-saturated models of $T$. Then we can decompose $M$ and $N$ in a direct sum of the form  $M=\tilde M\oplus  M_{t}^{00}$, $N=\tilde M\oplus N_{t}^{00}$, where $\tilde M\cong \mathbb D_{\mathbb G}$, $M_t^{00}\subseteq N_t^{00}$, and $M_{t}^{00}, N_{t}^{00}$ are torsion-free groups with the same 
indices as $\G$.
\elmm
\begin{proof} Let $M_{t}^{00}, N_{t}^{00}$ be defined as before, we have $M_{t}^{00}\subseteq N_{t}^{00}$. By Lemma \ref{partie_standard}, $N/N^{00}\cong \D $  
and $N^{00}$ is the kernel of the map $st$. 
We have $st(M_{t}^{00})=\mathbb D_{N_{\mathbb G}}= st(N_{t}^{00})$
and $M/M_{t}^{00}\cong \D_{\G}\cong N/N_{t}^{00}.$ The subgroup $M_{t}^{00}$ is pure in $M$ and so has a direct summand $\tilde M$ in $M$ (see e.g. \cite{Z}). Furthermore, this direct summand is isomorphic to $\D_{\G}$. Since $\tilde M+N_{t}^{00}$ contains the kernel of $st$ and $st(\tilde M+N_{t}^{00})=st(N)$, we have $\tilde M+N_{t}^{00}=N$ and it is a direct sum.
\end{proof}

\medskip
\noindent{\bf Proof of Theorem \ref{eqL}}:
\par We follow a strategy similar to  \cite{IKT}, using the decomposition given by Lemma \ref{proj_ss}.
\par First we make the following observation. Let $H$ be an $\aleph_{1}$-saturated model of $T$ and let $\tilde H$ be a direct summand as in Lemma \ref{proj_ss} (and so isomorphic to $\D_{\G}$). Denote the projection on $\tilde H$ by $[\quad]_1$. For $g, h\in H$, we will distinguish the following configurations for  $g<h$, setting $g_1=[g]_{1}, g_2=g-[g]_{1}, h_1=[h]_{1}, h_2=h-[h]_{1}$, so that, for instance $g=g_1+g_2$ is the decomposition of $g$ along the direct factors. We need also to specify whether $g_{2}\in (\frac{i}{n})$, $h_{2}\in (\frac{j}{n})$, $0\leq i,\,j\leq n-1$, with the convention that $H^{00}$ corresponds to the cut $(\frac{0}{n})$. Note that $st(g)=g_1+\frac{i}{n},\;st(h)=h_1+\frac{j}{n}$ for some $0\leq i, j<n$.
Since $n\in N_{\G}$, multiplication by $n$ is injective and it preserves the order relation except in the following case: suppose that $a\in (\frac{i}{n})^-$ and $b\in (\frac{i}{n})^+$, then $na>nb$.
However in that case, we have that if $a\in (\frac{i}{n})$ and $na\in H^{00}_{-}$, then $a\in (\frac{i}{n})^-$. Therefore for $a, b\in (\frac{i}{n})$, if $na\in H^{00}_{-}$ and $nb\in H^{00}_{+}$, then $a<b$. Using the fact that $na\in H^{00}_{-}$ iff $-na\in H^{00}_{+}$, we can always express the order relation within a cut, back in $H^{00}_{+}$.
\par Using this decomposition, we consider the different cases for $g<h$.   Let $st(g)=g_{1}+\frac{i}{n}$ and $st(h)=h_{1}+\frac{j}{n}$ with $0\leq i,j\leq n-1$, we have the following cases :
\begin{eqnarray}\label{order1}
  st(g)&<\;\;st(h)&{\rm\; and\; either}\\
\nonumber &&(i)\; st(g)\neq 0, st(h)\neq 0 {\rm\; and\; we\; distinguish\; the\; subcases\;whether\;}\\
\nonumber &&\; g_{1}\neq h_{1}\; {\rm and\;} g_{1}=h_{1}
{\rm (in\; case\;} g_{1}=h_{1}, {\rm \;it\; implies\; that\;} i\neq j);\\
\nonumber &&(ii)\; st(g)=0, {\rm namely\;} g_1=i=0, {\rm \;and\;} g\in H^{00}_{+} {\rm and\;} st(h)\neq 0;\\
\nonumber &&(iii)\; st(h)=0, {\rm namely\;} h_{1}=j=0 {\rm \;and\;} h\in H_{-}^{00} {\rm and\;} st(g)\neq 0;\\
\label{order2} st(g)&=\;\;st(h) &{\rm namely\;} g_1=h_1, i=j {\rm \;and\; either\;}\\
\nonumber &&(i)\; g_2\in (\frac{i}{n})^-, h_2\in (\frac{i}{n})^+, {\rm \;equivalently\;} ng_2\in H^{00}_{-}, nh_2\in H^{00}_{+},\\
\nonumber &&(ii)\; g_2\in (\frac{i}{n})^+, h_2\in (\frac{i}{n})^+, {\rm\; with\;} g_{2}<h_{2}, ng_2<nh_2 {\rm \;and\;} ng_2\in H^{00}_{+}, nh_2\in H^{00}_{+}\\
\nonumber &&(iii)\; g_2\in (\frac{i}{n})^-, h_2\in (\frac{i}{n})^-, {\rm \;with\;} g_{2}<h_{2}, ng_2<nh_2 {\rm \;with\;} ng_2\in H^{00}_{-}, nh_2\in H^{00}_{-}
\end{eqnarray}

\par Now let us examine the congruence relations $D_n$ in $H^{00}_{t}$. Recall that an element $u\in H_{+}^{00}$ is in $nH$ if and only if $u-v_{i,n}\in nH_{+}^{00}$ or $v_{i,n}-u\in nH_{+}^{00}$, $0\leq i\leq n-1$ (see Notation \ref{cosets}). Moreover if $u\in (\frac{j}{m})^+$, for some $1\leq j\leq m-1$, then $u\in nH$ if and only if $mu-v_{i,nm}\in nmH_{+}^{00}$ or $v_{i,nm}-mu\in nmH_{+}^{00}$, $1\leq i\leq nm-1$.
(Note that  $mu\notin nmH_{+}^{00}$ (otherwise $u$ would belong to $H_+^{00}$). Also, since $m\in N_{\G}$, we get that $mn\in N_{\G}$.)  
\par Now suppose that $u,\;v\in H_t^{00}$, we distinguish two cases. 
\par Either $u,\;v\in (\frac{i}{m})$ and so $v-u$ or $u-v$ belong to $H_+^{00}$. We get in case $v-u\in H_+^{00}$, that 
\begin{eqnarray}\label{cong1}
H\models D_n(u,v)&\mbox{iff} & H_+^{00}\models \bigvee_{i=0}^{n-1}\;D_n(v-u,v_{i,n})
\end{eqnarray}
and similarly when $u-v\in H_+^{00}$.
 \par Or $u\in (\frac{i}{m})$ and $v\in (\frac{j}{m})$, $0\leq i\neq j\leq m-1,$ and then we get:
\begin{eqnarray}\label{cong2}
H\models D_n(u,v)&\mbox{iff} &H^{00}_{+}\models\bigvee_{i=1}^{mn-1} D_{nm}(mu+v_{i,nm},mv) 
\end{eqnarray}
(or $D_{nm}(mu,mv+v_{i,nm})$  $1\leq i\leq n.m-1$).

\medskip
\par Let $M, N$ be two $\aleph_{1}$-saturated models of $T$ with $M\subset N$. Let $\tilde M\cong \D_{\G}$ as in Lemma \ref{proj_ss} and denote the projections on $\tilde M$ by $[\quad]_{1, M}, [\quad]_{1, N}$, and the projections on $M_{t}^{00},\;N_{t}^ {00}$ by  $[\quad]_{2, M},\;[\quad]_{2, N}$. 
Since both $M^{00}$ and $N^{00}$ are models of $T_{r\G,<}$, $M^{00}$ a $L'$-substructure of $N^{00}$, by Lemma \ref{modelecompleteTrg<}), we may choose the same coset representatives of $nN^{00}$ in $N^{00}$,  as those for $nM^{00}$ in $M^{00}$, $n\in \N$. As in Notation \ref{cosets}, we will denote those coset representatives by $u_{1}=v_{0,n},\cdots,u_{n_{\G}}, v_{1,n},\cdots, v_{n-1,n}$.
Note that for all $x,y\in M$ we have 
\begin{eqnarray*}
M\models D_n(x,y) &\mbox{iff} &\tilde M \models D_n([x]_{1,M}, [y]_{1,M}) \, \mbox{ and } \, M_{t}^{00}\models D_n([x]_{2,M}, [y]_{2,M}) \\
N\models D_n(x,y) &\mbox{iff} &\tilde M \models D_n([x]_{1,M}, [y]_{1, M}) \, \mbox{ and } \, N_{t}^{00}\models D_n([x]_{2,N}, [y]_{2,N}) 
\end{eqnarray*}

 Since the function $"-"$ is existentially definable in $T$, w.l.o.g. we may consider existential formulas where $"-"$ does not occur,

\par Let $\phi(x,{\bf y})$ be a $L'_\rho$-quantifier-free formula, where $-$ does not occur, and let ${\bf a} \in M$. Assume that there exists $b\in N$ such that $N\models \phi(b, {\bf a})$. Since $<$ is a total order we may assume that  $\phi(x, {\bf a})$ is of the form   

\begin{eqnarray*}
\phi(x, {\bf a}) &:=& \bigwedge_{j\in J} t_{j}({x},{\bf a})<t_{j}'({ x},{\bf a})\; \&\; \bigwedge_{\ell\in \Lambda} t_{\ell}({x},{\bf a})=t_{\ell}'({x},{\bf a})  \, \&\;   \\ 
&&  \bigwedge_{k\in \Delta} D_{n_k}(t_{k}({x},{\bf a}), t_{k}'({x},{\bf a})) \; \&\;   \bigwedge_{k\in \Delta'} \neg D_{n_k}(t_{k}({x},{\bf a}), t_{k}'({x},{\bf a})) 
\end{eqnarray*}
where $t_j, t_\ell, t_k, t'_j, t'_\ell, t'_k$ are $L_{mo}$-terms.

\par We can write a term $t({x}, {\bf a})$ in the form $t({x},{\bf a})=s({x})+ r({\bf a})$, where $s, r$ are $L_{mo}$-terms,  and if  ${\bf a}\in M_{t}^{00}$, then so is $r({\bf a})$. Consider all terms $t_j, t_\ell, t_k, t'_j, t'_\ell, t'_k$ written in this form with the corresponding indices, e.g. $t_j({x},{\bf a})=s_j({x})+ r_j({\bf a}).$ The projections are group morphisms, so we get $[s_{j}({x})+r_{j}({\bf a})]_{1}=[s_{j}({x})]_{1}+[r_{j}({\bf a})]_{1}=s_{j}([{x}]_{1})+r_{j}({\bf [a]}_{1})$, and similarly for $[\quad]_{2}$.

\par We have the following equivalences :
\begin{eqnarray} \label{proj_egalite_congruence}
N\models t_{\ell}({x},{\bf a})=t_{\ell}'({x},{\bf a}) &\mbox{iff} & 
 \D_{\G}\models s_{\ell}([{x}]_{1})+r_{\ell}({\bf [a]}_{1})= s'_{\ell}([{x}]_{1})+r'_{\ell}({\bf [a]}_{1}) \mbox{\;\;and} 
 \nonumber \\
&& N^{00}_{t}\models s_{\ell}([{x}]_{2})+r_{\ell}({\bf [a]}_{2})= s'_{\ell}([{x}]_{2})+r'_{\ell}({\bf [a]}_{2}) \\
N\models D_{n_k}(t_{k}({x},{\bf a}), t_{k}'({x},{\bf a})) & \mbox{iff}& \D_{\G}\models D_{n_k}( s_{k}([{x}]_{1})+r_{k}({\bf [a]}_{1}),  s'_{k}([{x}]_{1})+r'_{k}({\bf [a]}_{1})) \mbox{\;\;and} 
 \nonumber \\ 
&& N^{00}_{t}\models D_{n_k}( s_{k}([{x}]_{2})+r_{k}({\bf [a]}_{2}),  s'_{k}([{x}]_{2})+r'_{k}({\bf [a]}_{2}))
\end{eqnarray}
\par Using (\ref{cong1}), (\ref{cong2}), we can express the congruence conditions $D_{n_k}$ on the second projection, back in the subgroup $N^{00}$ in the following way.

\par Assume first that $s_{k}([{x}]_{2})+r_{k}({\bf [a]}_{2}),\;s'_{k}([{x}]_{2})+r'_{k}({\bf [a]}_{2}))\in (\frac{i}{m})$, 
then we get: 
\begin{eqnarray} \label{congruence1}
\nonumber N^{00}_{t}\models D_{n_k}( s_{k}([{x}]_{2})+r_{k}({\bf [a]}_{2}),  s'_{k}([{x}]_{2})+r'_{k}({\bf [a]}_{2}))&\mbox{iff}&\\
 N_+^{00}\models \bigvee_{i=0}^{n_{k}m-1}\;D_{n_{k}m}(ms_{k}([{x}]_{2})+mr_{k}({\bf [a]}_{2}),v_{i,n_{k}m}+ms'_{k}([{x}]_{2})+mr'_{k}({\bf [a]}_{2})))
\end{eqnarray}

\par Assume now that $s_{k}([{x}]_{2})+r_{k}({\bf [a]}_{2})\in (\frac{i}{m})$ and $s'_{k}([{x}]_{2})+r'_{k}({\bf [a]}_{2}))\in (\frac{j}{m})$, $0\leq i\neq j\leq m-1$, then we get:
\begin{eqnarray} \label{congruence2}
\nonumber 
N^{00}_{t}\models D_{n_k}( s_{k}([{x}]_{2})+r_{k}({\bf [a]}_{2}),  s'_{k}([{x}]_{2})+r'_{k}({\bf [a]}_{2}))&\mbox{iff} &\\
N_+^{00}\models\bigvee_{i=1}^{n_{k}m-1} D_{n_{k}m}(ms_{k}([{x}]_{2})+mr_{k}({\bf [a]}_{2})+v_{i,n_{k}m},ms'_{k}([{x}]_{2})+mr'_{k}({\bf [a]}_{2}))), 
\end{eqnarray}
(or $D_{n_{k}m}(ms_{k}([{x}]_{2})+mr_{k}({\bf [a]}_{2}),ms'_{k}([{x}]_{2})+mr'_{k}({\bf [a]}_{2}))+v_{i,n_{k}m})$).
\medskip
\par W.l.o.g., we may assume that $[b]_{1}\in M$ and that not only $[b]_{2}\in (\frac{i}{n})$, $0\leq i<n$, $n\in N_{\G}$, but all the components of the terms $[{\bf a}]_{2}$ also belong to a cut of the form $(\frac{j}{n})$, for some $0\leq j<n$.
Since $n\in N_{\G}$ and so multiplication by $n$ is injective, we obtain the following equivalence: 
\begin{eqnarray}\label{egalite}
\nonumber N^{00}_{t}\models s_{\ell}([{b}]_{2})+r_{\ell}({\bf [a]}_{2})= s'_{\ell}([{b}]_{2})+r'_{\ell}({\bf [a]}_{2})&\mbox{iff} &
\\
N^{00}\models n.s_{\ell}([{b}]_{2})+n.r_{\ell}({\bf [a]}_{2})= n.s'_{\ell}([{b}]_{2})+n.r'_{\ell}({\bf [a]}_{2})
\end{eqnarray}
(We could have used that it is the negation of being in strict order relation).
Recall that $st(s_{j}(b)+r_{j}({\bf a}))$ (respectively $st(s'_{j}(b)+r'_{j}({\bf a}))$) is determined by the type of $b$ (over $\D_{\G}$) and since the first projections of $s_{j}(b)+r_{j}({\bf [a]})$, respectively $s'_{j}(b)+r'_{j}({\bf [a]})$ belong to $\tilde M$, the cuts of the form $\frac{i}{n}$, $0\leq i\leq n-1$, $n\in N_{\G}$, to which their second projections belong to, are also determined by the type of $b$. Observe the following.
\par $(i)$ We saw in (\ref{order1}), (\ref{order2}) that the truth of an atomic formula of the form \\$s_{j}(b)+r_{j}({\bf a}))<s'_{j}(b)+r'_{j}({\bf a})$ is determined by on one hand the first projection of these terms, the cuts to which their second projections belong to and the order type of $n.(s_{j}([b]_{2})+r_{j}({\bf [a]_{2}})),\;n.(s'_{j}([b]_{2})+r'_{j}({\bf [a]_{2}}))$ in $N^{00}$. 
\par $(ii)$ We saw in (\ref{congruence1}), (\ref{congruence2}) that the truth of congruence condition of the form \\$D_{n_k}( s_{k}([{b}]_{2})+r_{k}({\bf [a]}_{2}),  s'_{k}([{b}]_{2})+r'_{k}({\bf [a]}_{2}))$ can also be expressed in $N^{00}$ using cosets representatives $v_{i,n_{k}.n}$, for some $i$.
\par $(iii)$ Finally in (\ref{egalite}), we reduce the truth in $N^{00}_{t}$ of $s_{\ell}([{b}]_{2})+r_{\ell}({\bf [a]}_{2})= s'_{\ell}([{b}]_{2})+r'_{\ell}({\bf [a]}_{2})$, to a statement in $N^{00}$.
\par  In each of the above three cases,  all the requirements put together can be expressed as a partial type of quantifier free formulas, say $\Phi (x, \bf a)$, and we have $N\models \Phi (b, \bf a).$
Now it remains to note that  $\Phi(x, \bf a)$ being satisfied by an element $b\in N$, implies that 
it is  finitely consistent in $M$, and so if $c$ realises it in $M$, then $M\models \phi(c,{\bf a})$. 
\par We use on one hand that $M/M^{00}\cong N/N^{00}\cong\D$ (Lemma \ref{partie_standard}), 
and that $M^{00}$ an $L'$-substructure of $N^{00}$ and so an elementary substructure by (Lemma \ref{eq-mrG}). Let $\Phi(y, [{\bf a}]_1)$ be the partial type of quantifier free formulas determined by $\Phi(x, \bf a)$ and satisfied by $[b]_1$, and let $\Psi(z, [{\bf a}]_2)$ be the partial type of quantifier free formulas determined by $\Phi(x, \bf a)$ and satisfied by $[b]_2$. We have $\mathbb D_{\mathbb G}\models  \Phi([{b}]_1, [{\bf a}]_1)$ and $N^{00}\models \psi([{b}]_2, [{\bf a}]_2)$. Since $M^{00}\subseteq_{ec} N^{00}$, there is ${d}\in M^{00}$ such that $M^{00}\models \psi({d}, [{\bf a}]_2)$. In case $d$ is of the form $n.d'$, with $n\in N_{\G}$, then there is $d'$ in $M$ belonging to a cut $(\frac{i}{n})$, $1\leq i\leq n-1$. Let ${c}=[{b}]_1+{d'}$. Then ${c}\in M$, $[{c}]_1=[{ b}]_1, [{ c}]_2={d'}$, and by 
$(\ref{order1})$, (\ref{order2}), $(\ref{congruence1}),\; (\ref{congruence2})$ and $(\ref{egalite})$, we conclude that $M\models \phi(c,{\bf a})$. This concludes the proof of Theorem \ref{eqL}. \hfill $\Box$
\medskip
\par Recall that a first-order theory is said to be NIP if no formula has the independence property. Y. Gurevich and P. Schmitt showed that the theories of ordered abelian groups are NIP (\cite{GS}).
\cor \label{completenessT} The theory $T$ is complete, and decidable whenever $N_{\G}$ is recursively enumerable. The theory $T$ is NIP.
\ecor
\begin{proof} 
We first show the completeness of $T$ by proving that two $\aleph_{1}$-saturated models of $T$ are elementarily equivalent. We follow the same strategy as in Theorem \ref{eqL}. Let $\si$ be a sentence true in some $\aleph_{1}$-saturated model $H$ of $T$. By Lemma \ref{proj_ss}, $H$ decomposes as $\tilde H$, isomorphic to $\D_{\G}$, and $H^{00}_{t}$. Since $T$ is model-complete, we may assume that $\si$ is of the form $\exists \bar x\;\phi(\bar x)$, where $\phi(\bar x)$ is a conjunction of 
formulas which are atomic or negated atomic. Let $\bf a$ be a tuple of elements of $H$ such that $H\models \phi(\bf a)$. By the proof of Theorem \ref{eqL}, this is equivalent to 
the requirements that two partial types are satisfied in respectively $\tilde H$ and $H^{00}$.  
Since $\tilde H$ is isomorphic to $\D_{\G}$, finite satisfiability of the first partial type does not depend on the model we are considering and for the second partial type, we apply the completeness of $T_{r\G,<}$ (Lemma \ref{eqTrg<}). In particular $T$ axiomatizes the theory of the standard model $\G$.  
 \par The model $\G$ is NIP since we can interpret any formula in the ordered abelian group $(G,+,0,1,<)$ where we can use the result of Y. Gurevich and P. Schmitt and therefore the complete theory $T$ is NIP. 
 \par Finally, the completeness of $T$ implies its decidability in case $T$ has an r.e. axiomatization, namely if $N_{\G}$ is r.e.
 \end{proof}

  \par Note that since $T$ is NIP, we know that in any definable group in a model of $T$, there is a smallest type-definable subgroup of bounded index \cite{Sh}. Recall that $n_{\G}$ is equal to the index $[H:nH]$, where $H$ is a model of $T$. We suspect that  $\bigcap_{\{n\in\mathbb N: n_{\G} \in \N\}} n H^{00}$ is the smallest type definable subgroup of $H$ of bounded index.

\medskip
\par Note that following the proof \cite[Example 2.10]{KS} that if $(G,+,<)$ is an ordered abelian group, then its theory is not strongly$^2$ dependent, one can show that the theory of $H$ is not strongly$^2$ dependent.
\medskip
\par Now, the ordered group of the integers has NIP, as well as the group of decimals $\D$ (cf. Corollary \ref{completenessT}). In view of the result of Bouchy-Finkel-Leroux recalled in the introduction on the decomposition of definable subsets in the expansion of the ordered additive group of real numbers with the integer part function, one can ask whether  
one could deduce that its theory has also NIP. 
 However, we haven't seen how to combine the two other results. We present instead a direct proof of NIP in the following proposition. A stronger version of this result is known (see \cite[Proposition 3.1]{dolich-goodrick}),  but we were not aware of that  unpublished result  at the time.
\prop The theory of the structure $(\R,+,-,\lfloor .\rfloor,<,0,1,\equiv_{n}; \;n\in \omega)$ is NIP. 
\eprop
\begin{proof} Set the language $\L:=\{+,-,\lfloor .\rfloor,<,0,1,\equiv_{n}; \;n\in \omega^*\}$. Weispfenning showed quantifier elimination in $\L$ (\cite[Theorem 3.1]{W}). Since NIP formulas are closed under boolean combinations (see for example  \cite[Lemma 2.9]{S}), it suffices to show that atomic formulas are NIP. 
\par Let $\psi(x, \bf{y})$ be an atomic $\L$-formula and let $(a_{\ell})_{\ell\in I}$ be an indiscernible sequence of elements in a saturated model $H$ of $T_{\mathcal{R}}$, and consider a tuple $\b\in H$, then let us show that the truth value of $\psi(a_{\ell},\b)$ is eventually constant. We can write $a_i=\lfloor a_i \rfloor +(a_i-\lfloor a_i \rfloor)=\lfloor a_i \rfloor-a_{i}^*$, using the notation $x^*:=x-\lfloor x \rfloor$. We have $\lfloor x_{1}+x_{2} \rfloor=\lfloor x_{1} \rfloor+\lfloor x_{2} \rfloor$, if $x_{1}^*+x_{2}^*<1,$ and $\lfloor x_{1}+x_{2} \rfloor=\lfloor x_{1} \rfloor+\lfloor x_{2} \rfloor+1,$ if $x_{1}^*+x_{2}^*\geq 1$. Depending on which subinterval of the form $[\frac{i}{n}, \frac{i+1}{n}[$, $0\leq i<n$, an element $x$ belongs to, we can express 
$\lfloor n.x\rfloor$ in terms of $\lfloor x\rfloor$ (and so $(n.x)^*$ in terms of $x^*$).
\par Remark that if $(a_{\ell})_{\ell\in I}$ is an indiscernible sequence, then $(\lfloor a_{\ell}\rfloor)_{\ell\in I}$ is an indiscernible sequence of the reduct $(\Z,+,-,0,\equiv_{n}; n\in \omega^*)$ and $(a_{\ell}-\lfloor a_{\ell}\rfloor)_{\ell\in I}$ is an indiscernible sequence of the reduct $(\R,+,-,0,1,<)$. We now indicate how to reduce to these two structures whose theories are NIP.
\par The atomic formulas are of the form $t(x,\bar y)=0,\; 0<t(x,\bar y), 0>t(x,\bar y), t(x,\bar y)\equiv_{n} 0$. We write $t(x,\bar y)$ as $n.x+s(\bar y)$ with $s(\bar y)$ an $\L$-term and w.l.o.g. we may assume that $n\in \N^\ast$.
\par Now $t(x,\bar y)=0$ iff $\lfloor t(x,\bar y)\rfloor=0$ and $t(x,\bar y)-\lfloor t(x,\bar y)\rfloor=0$. \par Also we have  
\[
\lfloor n.x+s(\bar y)\rfloor= \left\{
\begin{array}{lcl}
        \lfloor n.x\rfloor+\lfloor s(\bar y)\rfloor &\mbox{if}& (n.x)^*+s(\bar y)^*<1 \\
         \lfloor n.x\rfloor+\lfloor s(\bar y)\rfloor+1 &\mbox{if}&(n.x)^*+s(\bar y)^*\geq 1\\
         \end{array}
\right.
\]
\par So, we get $t(x,\bar y)=0$ iff 
\[
\left\{
\begin{array}{lc}
        \lfloor n.x\rfloor+\lfloor s(\bar y)\rfloor=0\;\&\; (n.x)^*+s(\bar y)^*=0&\mbox{or} \\
         \lfloor n.x\rfloor+\lfloor s(\bar y)\rfloor+1=0\;\&\; (n.x)^*+s(\bar y)^*=1 &\\
         \end{array}
\right.
\]
\par Now $t(x,\bar y)>0$ iff $\lfloor t(x,\bar y)\rfloor>0$ or ($\lfloor t(x,\bar y)\rfloor=0\;\&\;t-\lfloor t\rfloor>0$.
\par So, we get $t(x,\bar y)>0$ iff 
\[
\left\{
\begin{array}{lrrl}
        \lfloor n.x\rfloor+\lfloor s(\bar y)\rfloor>0&\mbox{if}& (n.x)^*+s(\bar y)^*<1&\mbox{or}\\
         \lfloor n.x\rfloor+\lfloor s(\bar y)\rfloor+1>0&\mbox{if}& (n.x)^*+s(\bar y)^*\geq 1&\mbox{or}\\
          \lfloor n.x\rfloor+\lfloor s(\bar y)\rfloor=0\;\&\; (0<(n.x)^*+s(\bar y)^*<1)&&\mbox{or}\\
         \lfloor n.x\rfloor+\lfloor s(\bar y)\rfloor+1=0\;\&\; (1<(n.x)^*+s(\bar y)^*). &&
         \end{array}
\right.
\]
\par Finally $t(x,\bar y)\equiv_{n} 0$ iff $t(x,\bar y)=\lfloor t(x,\bar y)\rfloor$ and $\lfloor t(x,\bar y)\rfloor\equiv_{n}0$.
Again we transform $\lfloor t(x,\bar y)\rfloor$ and we get $t(x,\bar y)\equiv_{n} 0$ iff
\[
\left\{
\begin{array}{lrrl}
        (n.x)^*+s(\bar y)^*=0\;\&\; (\lfloor n.x\rfloor + \lfloor s(\bar y)\rfloor\equiv_{n}0)&\mbox{if}&(n.x)^*+s(\bar y)^*<1&\mbox{or}\\
  
  (n.x)^*+s(\bar y)^*+1=0\;\&\;\lfloor n.x\rfloor + \lfloor s(\bar y)\rfloor+1\equiv_{n}0&\mbox{if}&(n.x)^*+s(\bar y)^*\geq 1
         \end{array}
\right.
\]
From that analysis together with the remark above, we get that the truth value of an atomic formula $\psi(a_{\ell},\bf{b})$ is eventually constant. \end{proof}

\section{Quantifier elimination for all decimals}\label{qe}

In \cite{belair-point2014} we showed quantifier elimination for the structure $\mathbb D$ in a language with extra unary function symbols, and mentioned that these extra symbols could be eliminated to get quantifier elimination down to the language $L=\{+, -, 0, <\}$. In this section we set $\mathbb G=\mathbb D$ and we show how to eliminate the extra symbols. Note that given the connection with the circular group formalism, this also follows from the quantifier elimination established by Lucas (see \cite{lucas_ddg}, \cite{leloup}).

Note that  for $\mathbb G=\mathbb D$, the set of torsion elements $\{ c_n : n\ge 1\}$ is equal to $\{ \frac{i}{n} : n\ge 2, 1\le i<n\}$, and our convention yields $\rho^{\mathbb D}_{\frac{i}{n}}=\frac{i}{n}.$ Also, since we now have all the possible torsion, we get $H^{00}_t=H^{00}$. 

\par For each natural number $n, n\ne 0, 1$, consider the function $f_n:\mathbb D\to \mathbb D$ defined by $f_n(x)=$ the smallest $y$ such that $ny=x$, that is $f(x)=\frac{x}{n}$. Note that in $\mathbb D$ 
we have  $f_n(x)=z\leftrightarrow (nz=x\; \&\; \bigwedge_{0\le i\le n-1} z\le z+\frac{i}{n})$. In \cite{belair-point2014}, we used the relation $\frac{i}{n}=2if_n(\frac{1}{2})$ to reduce the set of constant symbols $\rho_n$ to just $\rho_{\frac{1}{2}}$.

For the sake of easier readability in this section, we will use the symbols $\frac{i}{n}$ as constant symbols $\rho_m$ in the language $L_\rho$.

We let $L_{\rho,f}=L_{\rho}\cup \{ f_n : n\ge 2\}$. Since $f_n$ is definable in $\mathbb D$ for the language $L_\rho$, it follows that the theory of $\mathbb D$ is also model-complete in $L_{\rho,f}$. It turns out that in $L_{\rho,f}$ our axioms for $\mathbb D$ can be formulated as universal axioms as we showed in \cite{belair-point2014}, and thus quantifier elimination in $L_{\rho,f}$ follows. For the benefit of the reader we reproduce the axioms here.

\begin{definition} Let 
$\mathcal T$ be the following set of axioms in $L_{\rho,f}$. 
\begin{enumerate}
\item[(1)] The axioms for an abelian group;
\item[(2)] the relation $<$ is a strict order relation with $0$ as a minimum; 
\item[(3)] $\forall x \forall y [(x < y \rightarrow (x < x+f_2(y-x) < y)) \& (x\ne 0\rightarrow x< x+f_2(-x)\;)]$;
\item[(4)] $\forall x\forall y (x<y \rightarrow -y<-x)$;
\item[(5)] $\forall x \left(\frac{1}{2}\ne 0 \; \& \; 2\frac{1}{2}=0\; \&\;  (2x=0 \rightarrow (x=0 \vee x=\frac{1}{2}))\right)$.
\end{enumerate}
The following three axioms for each natural number  $n\ne 0, 1$.
\begin{enumerate}
\item[(6)]  $\forall x\forall y \; (nf_n(x)=x \ \& f_n(x)\le x \ \& \  ( ny=x\rightarrow f_n(x)\le y))$ ;
\item[(7)] $\forall x \left(\bigwedge_{1\le i<n-1} (\frac{i}{n}\ne 0 \; \&\; n\frac{i}{n}=0) \; \& \; \left(nx=0\rightarrow \bigvee_{0\le i<n-1} x=\frac{i}{n}\right)\right)$ ;
\item[(8)] $\frac{1}{n}<\frac{2}{n}<\ldots<\frac{n-1}{n}$.
\item[(9)] For all natural numbers  $m, n, m\prime, n\prime \ne 0,1$, the following axiom :
$$\bigwedge_{i,j,i\prime, j\prime, 0<\frac{i}{n}+\frac{j}{m}<1}  
\left((\frac{i\prime}{n\prime}\le x\le \frac{i}{n} \& \frac{j\prime}{m\prime}\le y\le \frac{j}{m}) \rightarrow  \frac{i\prime}{n\prime} + \frac{j\prime}{m\prime} \le x+y\le \frac{i}{n} + \frac{j}{m}\right)$$ 
where $0\le i\prime<n\prime, 0\le i<n, 0\le j\prime<m\prime, 0\le j<m$  and $\frac{0}{k}:=0$.
\item[(10)] For all natural number $n\ge 3$, the following axiom  :
$$ (x\le \frac{1}{n} \ \& \ y\le\frac{1}{n} \ \& \ z\le \frac{1}{n} ) \rightarrow \left( (x\le y \leftrightarrow x+z\le y+z) \ \& \ (x\le x+y) \right).$$
\end{enumerate}
\end{definition}

Recall the axiomatization $T$ of $\mathbb D$ in the language $L_\rho$, and the definition of  $\frac{1}{n}$ in $\mathbb D$ : $z=\frac{1}{n} \leftrightarrow z\ne 0 \; \& \; nz=0 \; \& \; \wedge_{2\le j\le n-1} z< j z$. Let $T'$ be the axiomatization of $\mathbb D$ in the language $L$ obtained from $T$, by replacing each constant symbol $\frac{i}{n}$ using the definition of $\frac{1}{n}$.  For example, $x\le\frac{2}{3}$ becomes $\exists u  (u\ne 0\;\& \;3u=0\;\&\; u<2u\; \&\; x\le u)$. 
We now show how to eliminate the symbols $f_n$ and $\frac{i}{n}$ in the {\em theorems} of $T'$.

\begin{lemma} \label{elimine_i/n}
For every $n, i\in\mathbb N, n\ge 2, 1\le i<n$, there exist quantifier-free $L$-formulas $\theta_{i,n}(x), \varphi_{n,i} (x)$ such that in every model of $T'$ we have that for all $x\ne 0$, $x=\frac{i}{n} \leftrightarrow \theta_{i,n}(x)$ and $x<\frac{i}{n} \leftrightarrow \varphi_{i,n}(x)$ hold.
\end{lemma}

\begin{proof}  
It suffices to check in $\mathbb D$. Suppose $x\ne 0$. First recall that $x=\frac{1}{2} \leftrightarrow 2x=0$ and $x<\frac{1}{2} \leftrightarrow x<2x$. Let's consider the relation $x=\frac{i}{n}$. We may assume that $i,n$ are coprime, and then we have $x=\frac{i}{n} \leftrightarrow nx=0 \, \& \, \sigma(1) x<\ldots < \sigma(n-1)x$, where $\sigma$ is the permutation of  $1,\ldots, n-1$ which corresponds to the correct ordering of  $\frac{i}{n}, 2\frac{i}{n}, \ldots, (n-1)\frac{i}{n}$ modulo $1$. Now consider 
the relation $x<\frac{i}{n}.$ We claim that  $x<\frac{1}{n} \leftrightarrow  x<2x<\ldots < nx $. Certainly $x<\frac{1}{n} \rightarrow x<2x<\ldots < nx $ ( from axiom (9), when $n\ge 3$).
For the reciprocal, with $n\ge 3$, if $x<2x<\ldots < nx $, then by induction we have $x<\frac{1}{n-1}$. If $\frac{1}{n} \le x$, then $\frac{n-1}{n}\le (n-1)x$ and $nx<\frac{1}{n-1}$ (computing modulo $1$). Since $\frac{n-1}{n}>\frac{1}{n-1}$, when $n\ge 3$, we would obtain $nx<(n-1)x$ and a contradiction. Hence  $x<\frac{1}{n}$ must hold, as wanted.

Now let's consider the relation  $x<\frac{i}{n}$, $1<i<n, n\ge 3$. It suffices to show, by induction on $n$, that we can subdivide evey interval $[\frac{i}{n}, \frac{i+1}{n}[, 0\le i<n$, in subintervals whose extremities belong to  $\cup_{k=2}^n \{0, \frac{1}{k}, \ldots, \frac{k-1}{k}\}$ and such that within each such subinterval  $I$ we have $\sigma_I(1)x<\sigma_I(2)x<\cdots<\sigma_I(n-1)x$ for some permutation  $\sigma_I$ of $1, \ldots, n$, and such that $\sigma_I\ne \sigma_J$, when $I\ne J.$ The cases $n=2,3$ follow from the previous remarks. Let's consider the  induction step from $n$ to $n+1$. 

First note that the interval $]0, \frac{1}{n+1}[$ is defined by  $x<2x<3x<\cdots<(n+1).x$, and  $]\frac{n}{n+1}, 1[$ is defined by  $x>2x>3x>\cdots>(n+1).x$. In an interval $[\frac{k-1}{n+1}, \frac{k}{n+1}[$, $0< k\leq n$, we have that  $(n+1).x$ modulo $1$ is equal to  $(n+1).x-(k-1)$. Let's compare it with $m.x-\ell$, where $1\leq m\leq n$ and $0\leq \ell<m$. Note that if $k-1\leq (n+1).x\leq k$ and $\ell\leq m.x\leq \ell+1$, then $\ell\leq k-1$. Now, if $(n+1).x-(k-1)=m.x-\ell$, then $(n+1-n).x=k-1-\ell$ and so $x=\frac{k-1-\ell}{n+1-m}$. So that if there is no subdivision inside our interval, we have either  $(n+1).x-(k-1)< m.x-\ell$ for all $x$, or $(n+1).x-(k-1)> m.x-\ell$ for all $x$, and so coming back modulo $1$, either $(n+1).x< m.x$, or $(n+1).x> m.x$.

We are left to check that we could subdivide any two intervals  $[\frac{k-1}{n+1}, \frac{k}{n+1}[$, $[\frac{k'-1}{n+1}, \frac{k'}{n+1}[$, $1\leq k\neq k'\leq n+1$ in such way that they are associated to two distinct permutations of $\{1,\cdots,n+1\}$. By the pigeonhole principle, we have that in the remaining  $n-1$ intervals we have exactly one point of the form  $\frac{k}{n}$, $0<k<n$ (there cannot be two in the same interval). Hence, by induction, we can distinguish an interval of type  $[\frac{k-1}{n+1}, \frac{k}{n+1}[$, $1\leq k\leq n+1$, from an interval not adjacent to it. We are left to be able to distinguish adjacent intervals. Let's consider $\frac{k}{n+1}$. For sufficiently small positive $\varepsilon$, we can arrange  $(n+1).(\frac{k}{n+1}-\varepsilon)$  (mod $1$)  to be near $1$ whereas $(n+1).(\frac{k}{n+1}+\varepsilon)$ (mod $1$) is near $0$. We obtain that  $x< (n+1).x $ to the left of   $\frac{k}{n+1}$, and $x> (n+1).x $ to the right of $\frac{k}{n+1}$, which is sufficent to distinguish the two intervals which are adjacent in  $\frac{k}{n+1}$.
\end{proof}

\begin{theorem} \label{eq-L} 
The theory $T'$ admits quantifier elimination in the language $L$.
\end{theorem}

As before, it follows from the following lemma. 
\begin{lemma} \label{lem-eqL}
Suppose $M,N$ models of $T'$ and $A$ a common substructure, $A\subseteq M, A\subseteq N$. Let  $\phi({\bf y}, x)$ be a quantifier- free $L$-formula and  ${\bf a}\in A$, and suppose there exists $b\in M$ such that $M\models \phi({\bf a}, b)$. Then there exists $c\in N$ such that $N\models \phi({\bf a}, c)$.
\end{lemma}
\begin{proof} 
We can assume $A, M,N$ are $\aleph_1$-saturated. We will use a similar argument and notation as in theorem \ref{eqL}. Let $[\quad]_{1, M}, [\quad]_{1, N}, [\quad]_{2, M}, [\quad]_{2, N}$ be the projections. By lemma \ref{elimine_i/n}, we have $[a]_{1, M}= [a]_{1, N}$, and
 $a\in M^{00} \leftrightarrow a\in N^{00}$, for all $a\in A$. Then 
$A^{00}$ is a pure subgroup of $A$. Since $A$ is pure-injective (being $\aleph_1$-saturated), $A^{00}$ is a direct summand of $A$ (\cite{Z}). Let $\tilde A$ be a subgroup of $A$ such that $A=\tilde A \oplus A^{00}$, and $\tilde M$ be a subgroup of $M$ maximal such that $\tilde A\subseteq \tilde M$ and $\tilde M \cap M^{00}=\{0\}.$ Then $M=\tilde M \oplus M^{00}$. Indeed,  first note that $\tilde M$ is divisible, and contains all $\frac{i}{n}$  by maximality and since $M^{00}$ is torsion free.  Hence, for $x\in M$, if $x\not\in \tilde M$, then there exist a positive integer  $n$ and $m_1\in \tilde M$ such that $m_1+nx\in M^{00}$.  Let $b=m_1+nx$ and  $m\prime \in\tilde M, b\prime\in M^{00}$ be such that $nm\prime=m_1, nb\prime=b$.  We get $nx=n(b\prime-m\prime)$, so  $x=-m\prime+\frac{i}{n}+b\prime$, for some  $\frac{i}{n}$, and $-m\prime+\frac{i}{n}\in \tilde M, b\prime\in M^{00}$, as wanted. Similarly we get a subgroup  $\tilde N$ such that $\tilde A \subseteq \tilde N$ and $N=\tilde N\oplus N^{00}$.  It follows from these remarks that in the construction for theorem \ref{eq-L}, we can assume that the projections are such that for all   $a\in A$ we have $[a]_{1, M}= [a]_{1, N}$ and $[a]_{2, M}= [a]_{2, N}$. 
Now using lemma \ref{eq-mrG} instead of lemma \ref{modelecompleteTrg<}, the calculations in theorem \ref{eqL} apply directly (and simplify since we no longer have to worry about cuts) to check that if there is $b\in M$ such that $M\models \phi({\bf a}, b)$, then there is $c\in N$ such that $N\models \phi({\bf a}, c)$. \end{proof}

\par{\bf Acknowledgements:}
The first author thanks the Logic Group of the Department of Mathematics, University of California, Berkeley, and the second author thanks also the MSRI, for their hospitality during the spring of 2014. The first author was partially supported by  NSERC and the second author by FRS-FNRS.


\end{document}